\documentclass[final,1p,times]{elsarticle}

\usepackage{amssymb}
\usepackage{amsthm}
\usepackage{amsmath,amssymb,amsopn,amsfonts,mathrsfs,amsbsy,amscd}
\usepackage{longtable}

\newcommand{\bbC}{\mathbb{C}}

\newcommand{\g}{\mathfrak{g}}

\newcommand{\prs}{\langle\;,\;\rangle}

\newcommand{\too}{\longrightarrow}

\newcommand{\esp}{\quad\mbox{and}\quad}

\newcommand{\R}{\mathbb{R}}
\newcommand{\G}{{\mathfrak{g}}}

\newcommand{\h}{{\mathfrak{h}}}
\newcommand{\ad}{{\mathrm{ad}}}

\newcommand{\ric}{{\mathrm{ric}}}
\newcommand{\Li}{\mathrm{L}}

\newcommand{\na}{\nabla}

\newcommand{\al}{\alpha}

\newcommand{\e}{\epsilon}

\newcommand{\la}{\lambda}

\newtheorem{theo}{Theorem}[section]
\newtheorem{pr}{Proposition}[section]
\newtheorem{Le}{Lemma}[section]

\begin{document}

\begin{frontmatter}
	
	
	
	
	\title{   Kenmotsu  Lie groups}
	
	\author[label1]{ Mohamed Boucetta}

	\address[label1]{Universit\'e Cadi-Ayyad\\
		Facult\'e des sciences et techniques\\
		BP 549 Marrakech Maroc\\e-mail: m.boucetta@uca.ac.ma
	}

	
	
	
	\begin{abstract} Kenmotsu manifolds constitute an important subclass of the class of contact Riemannian manifolds. In this note, we determine entirely connected and simply-connected Lie groups having a left invariant Kenmotsu structure. We show also that these Lie groups are Einstein Riemannian manifolds. 
		
	\end{abstract}
	
	\begin{keyword} Kenmotsu manifolds \sep Lie groups \sep Lie algebras   \sep   \sep 
		\MSC 53C25  \sep \MSC 22E25 
		
		
	\end{keyword}
	
\end{frontmatter}







\section{Introduction and main result}\label{section1}

Contact geometry is a crucial area in modern differential geometry. K. Kenmotsu introduced a distinctive class of contact Riemannian manifolds which later became known as Kenmotsu manifolds. Kenmotsu demonstrated that a Kenmotsu manifold is locally a warped product $I\times_f N$
where $I$
is an interval and $N$
is a K\"ahler manifold, with the warping function $f(t)=ke^t$
where $k$ 
is a non-zero constant. Over the years, many researchers have explored the geometry of Kenmotsu manifolds \cite{I, II, III, IV, V}.

An almost contact metric manifold is a $(2n+1)$-dimensional Riemannian manifold $(M,h)$ that admits a $(1, 1)$-tensor field $\phi$ and a characteristic vector field $\xi$ satisfying the following relations
\begin{equation}\label{almost} 
	\phi^2(X)=-X+\eta(X)\xi,\;
	h(\phi(X),\phi(Y))=h( X,Y)-\eta(X)\eta(Y),
\end{equation}for any  $X,Y\in TM$ where $\eta(X)=h( X,\xi)$. If the Levi-Civita connection $\na$ of an almost contact metric manifold $(M, g, \phi, \xi)$
satisfies
\begin{equation}\label{kenmostu}
	\na_X(\phi)Y=h(\phi X,Y)-\eta(Y)\phi(X)
\end{equation}for any $X,Y\in TM$, then it is called a Kenmotsu  manifold \cite{KENMOTSU}.

A {\it Kenmotsu Lie group} is a Lie group $G$ endowed with a Kenmotsu structure $(h,\phi,\xi)$ where $(h,\phi,\xi)$ are left invariant. In this Note, we give a complete description of Kenmotsu Lie groups by proving the following theorem.

\begin{theo}\label{main} Let $(G,h,\phi,\xi)$ be a connected and simply-connected Kenmotsu Lie group of dimension $2n+1$. Then there exists $(\la_1,\ldots,\la_n)\in\R^n$ such that  $(G,h,\phi,\xi)$ is isomorphic to $(\R\times\mathbb{C}^n,h_0,\phi_0,\xi_0)$ where the group product is given by
	\[ (t_1,(z_1,\ldots,z_n)).
	(t_2,(z_1',\ldots,z_n'))=(t_1+t_2,
	(z_1+z_1'e^{-t_1(1+\imath\la_1)},\ldots,z_n+z_n'e^{-t_1(1+\imath\la_n)})), \]
	$h_0=dt^2+e^{2t}\mathrm{can}$, $\phi_0(\frac{\partial}{\partial t})=0$, the restriction of $\phi_0$ to $\mathbb{C}^n$ is the canonical complex structure of $\mathbb{C}^n$,  $\xi_0=\frac{\partial}{\partial t}$ and $\mathrm{can}$ is the canonical metric of $\mathbb{C}^n$. Moreover, $(G,h)$ is Einstein and its Ricci curvature satisfies $\ric=-2n h$.
\end{theo}

Beside the introduction, this Note contains one section devoted to the proof of Theorem \ref{main}.

\section{Proof of Theorem \ref{main}}\label{section2}

In this section, we aim to prove Theorem \ref{main}. To do so, we proceed by establishing many intermediary results.

We start by describing our model of Kenmotsu Lie groups. Let $\la=(\la_1,\ldots,\la_n)\in\R^n$. Consider the Lie group $G_\la$ whose underlying manifold is $\R\times\bbC^n$ and the multiplication is given by
\begin{equation}\label{mult} (t_1,(z_1,\ldots,z_n)).
	(t_2,(z_1',\ldots,z_n'))=(t_1+t_2,
	(z_1+z_1'e^{-t_1(1+\imath\la_1)},\ldots,z_n+z_n'e^{-t_1(1+\imath\la_n)})). \end{equation}
The Lie algebra $\G_\la$ of $G_\la$ can be identified as a vector space to $T_{(0,0)}G_\la=\R\times\bbC^n$. For any $(a,u)\in\G_\la$, we denote by $(a,u)^\ell$ the left invariant vector field on $G_\la$ associated to $(a,u)$.
We have
\begin{align*} (a,u)^l(t,z)&=\frac{d}{ds}_{|s=0}(t,z).(sa,su)\\&=
	\frac{d}{ds}_{|s=0}(t+sa,
	(z_1+su_1e^{-t(1+\imath\la_1)},\ldots,z_n+su_ne^{-t(1+\imath\la_n)}))\\&=
	(a,u_1e^{-t(1+\imath\la_1)},\ldots,u_ne^{-t(1+\imath\la_n)})) \end{align*}
Denote by $e_0=1$ a generator of $\R$,  $(e_1,\ldots,e_n)$ the canonical basis of $\bbC^n$ as a complex vector space and put $f_j=\imath e_j$. We have, for any $j\in\{1,\ldots,n\}$,
\[ e_0^\ell=\frac{\partial}{\partial t},\; e_j^l=
e^{-t}\left(\cos(\la_j t)\frac{\partial}{\partial x_j}-\sin(\la_j t)\frac{\partial}{\partial y_j}\right)\esp f_j^l=
e^{-t}\left(\sin(\la_j t)\frac{\partial}{\partial x_j}+\cos(\la_j t)\frac{\partial}{\partial y_j}\right),
\]where $(t,x_1,y_1,\ldots,x_n,y_n)$ is the coordinates system associated to the basis $(e_0,e_1,f_1,\ldots,e_n,f_n)$. We have obviously that, for any $j,k\in\{1,\ldots,n\}$,
\begin{equation}\label{bracket} 
	[e_j^\ell,e_k^\ell]=[f_j^\ell,f_k^\ell]=[e_j^\ell,f_k^\ell]=0,\;
	[e_0^\ell,e_j^\ell]=-e_j^\ell-\la_j f_j^\ell\esp [e_0^\ell,f_j^\ell]=-f_j^\ell+\la_j e_j^\ell.
\end{equation}
Let $h_0$ be the left invariant metric on $G_\la$ for which $(e_0^\ell,e_1^\ell,f_1^\ell,\ldots,e_n^\ell,f_n^\ell)$ is orthonormal. From the relations
\[ \frac{\partial}{\partial x_j}=
e^{t}\left(\cos(\la_j t)e_j^\ell+\sin(\la_j t)f_j^\ell\right)\esp \frac{\partial}{\partial y_j}=
e^{t}\left(-\sin(\la_j t)e_j^\ell+\cos(\la_j t)f_j^\ell\right) \]we deduce that
\[ h_0=dt^2+e^{2t}\sum_{j=1}^n\left(dx_j^2+dy_j^2\right). \]
Finally, put $\xi_0=\frac{\partial}{\partial t}=e_0^\ell$ and denote by $\phi_0$ the left invariant $(1,1)$-tensor field given by $\phi_0(e_0^\ell)=0$, $\phi_O(e_j^\ell)=f_j^\ell$ and $\phi_O(f_j^\ell)=-e_j^\ell$ for any $j\in\{1,\ldots,n\}$.

Now we establish a useful lemma.
\begin{Le}\label{le} Let $V$ be a real vector space of dimension $2n$, $\prs$ a scalar product on $V$,  $J:V\too V$ a skew-symmetric isomorphism such that $J^2=-\mathrm{Id}_V$ and $D:V\too V$ an endomorphism such that $[J,D]=0$ and $D+D^t=-2\mathrm{Id}_V$ where $D^t$ is the adjoint of 
	$D$ with respect to $\prs$. Then there exists $(\la_1,\ldots,\la_n)\in\R^n$ and an orthonormal basis $(e_1,Je_1,\ldots,e_n,Je_n)$ of $V$ such that, for any $j\in\{1,\ldots,n\}$,
	\[ De_j=-e_j-\la_j Je_j\esp DJe_j=-Je_j+\la_j e_j. \]
	
\end{Le}

\begin{proof} The real vector space $V$ has also a structure of complex vector space given by
	\[ (a+\imath b).v=av+bJv,\quad a,b\in\R, v\in V, \]and since $[J,D]=0$ then $D$ is a complex endomorphism. Moreover, the symmetric bilinear form given by
	\[ h(u,v)=\langle u,v\rangle+\imath\langle u,Jv\rangle \]is a Hermitian metric and $D^t$ is also the adjoint of $D$ with respect to $h$.
	From the relation $D+D^t=-2\mathrm{Id}_V$, we deduce that
	$[D,D^t]=0$, i.e., $D$ is normal and hence it  is diagonalizable. Thus there exists $(\al_1,\la_1,\ldots,\al_n,\la_n)$ and a $h$-orthonormal complex basis $(e_1,\ldots,e_n)$ such that, for any $j\in\{1,\ldots,n\}$,
	\[ De_j=(\al_j-\imath \la_j)e_j=\al_j e_j-\la_j Je_j. \]
	But $A=D+\mathrm{id}_\h$ is skew-symmetric and 
	\[ \langle Ae_j,e_j\rangle=\langle (\al_j+1)e_j-\la_j Je_j,e_j\rangle=(\al_j+1)\langle e_j,e_j\rangle=0 \]and hence $\al_j=-1$ which completes the proof.
\end{proof}

Let $G$ be a Lie group and $\G$ its Lie algebra identified with the space of left invariant vector fields on $G$. For any left invariant metric $h$ on $\G$, we denote by $\prs$ its restriction to $\G$. The Levi-Civita product $\Li:\G\times\G\too\G$ is the restriction of the Levi-Civita connection $\na$ of $h$ to $\G$ and it is given by
\begin{equation}
	2\langle \Li_XY,Z\rangle=\langle[X,Y],Z\rangle+\langle[Z,X],Y\rangle+\langle[Z,Y],X\rangle,\quad X,Y,Z\in\G.
\end{equation} 
This product is characterized by the facts that, for any $X,Y\in\G$, \begin{equation}\label{Levi}\Li_X+\Li_X^t=0 \esp  \Li_XY-\Li_YX=[X,Y].\end{equation}
A left invariant almost complex structure on $G$ is entirely determined by an isomorphism $J:\G\too\G$ such that $J^2=-\mathrm{id}_\G$. When $h$ is Hermitian with respect to $J$ and $\na J=0$, $(G,h,J)$ is called a K\"ahler Lie group and $(\G,\prs,J)$ is called a K\"ahler Lie algebra. Note that the condition $\na J=0$ is equivalent to $[\Li_X,J]=0$ for any $X\in\G$.

Recall that a Kenmotsu Lie  group is a Lie group $G$ endowed with a left invariant metric $h$, a left invariant $(1,1)$-tensor field $\phi:TG\too TG$ and a left invariant vector field $\xi$ satisfying \eqref{almost} and \eqref{kenmostu}.  Since $(h,\phi)$ are left invariant, they are entirely determined by their restrictions $(\prs,\phi)$ to $\G$.  
The relations \eqref{almost} and \eqref{kenmostu} are equivalent to
\begin{equation}\label{kenmostubis} \begin{cases}
		\langle \xi,\xi\rangle=1,\;	\phi^2(X)=-X+\eta(X)\xi,\;
		\langle\phi(X),\phi(Y)\rangle=\langle X,Y\rangle-\eta(X)\eta(Y),\\
		\Li_X\phi(Y)-\phi(\Li_XY)=\langle\phi(X),Y\rangle\xi-\eta(Y)\phi(X),
\end{cases} \end{equation}for any $X,Y\in\G$, where $\eta$ is the left invariant 1-form given by $\eta(X)=\langle X,\xi\rangle$.

\begin{pr}\label{pr} For any $X,Y\in\G$,
	
	\begin{equation}\begin{cases} \phi(\xi)=0,\; \eta\circ\phi=0,\; \phi^t+\phi=0,\\
			\Li_X\xi=X-\eta (X) \xi,\; [\Li_\xi,\phi]=0,\;  \prec\eta,\Li_XY\succ=-\langle X,Y\rangle+\eta(X)\eta(Y),\end{cases} \end{equation}where $\phi^t$ is the adjoint of $\phi$ with respect to $\prs$.
	In particular, $d\eta=0$.
\end{pr}

\begin{proof} The relations in the first line of \eqref{kenmostubis} can be checked easily. From the last relation in \eqref{kenmostubis} and \eqref{Levi}, we have $\phi(\Li_X\xi)=\phi(X)$ and $\eta(\Li_X\xi)=\langle\Li_X\xi,\xi\rangle=0$.
	Thus
	\[ \Li_X\xi=X-\eta (X) \xi. \]
	Since $\phi(\xi)=0$, we deduce from the last relation in \eqref{kenmostubis} that
	$[\Li_\xi,\phi]=0$. On the other hand,
	\begin{align*}
		\prec\eta,\Li_XY\succ&=\langle \Li_XY,\xi\rangle=-\langle Y,\Li_X\xi\rangle\\
		&=-\langle X,Y\rangle+\eta(X)\eta(Y).
	\end{align*}
	Finally, for any $X,Y\in\G$,
	\[ d\eta(X,Y)=-\eta([X,Y])=-\eta(\Li_XY)+\eta(\Li_YX)=0. \]
\end{proof}
Put $\h=\xi^\perp=\ker\eta$. From the relation $d\eta=0$, we have $[\g,\g]\subset\h$ and hence $\h$ is an ideal of $\G$. Moreover, since $\eta\circ\phi=0$, $\phi$ leaves $\h$ invariant and induces a    isomorphism $J$ on $\h$. Denote by $D$ (resp. $\prs_\h$) the restriction of $\ad_\xi$ (resp. $\prs$) to $\h$.
\begin{pr} $(\h,\prs_\h,J)$ is a K\"ahler Lie algebra and $D$ is an invertible derivation of $\h$ satisfying $[D,J]=0$ and $D+D^t=-2\mathrm{Id}_\h$.
\end{pr}

\begin{proof} We have obviously from \eqref{kenmostubis} that $J^2=-\mathrm{Id}_\h$ and $h$ is Hermitian with respect to $J$. Let us compute the Levi-Civita product $\Li_X^\h Y$ of $(\h,\prs_\h)$. For any $X,Y,Z\in\h$,
	\begin{align*}
		2\langle \Li_X^\h Y,Z\rangle_\h=
		2\langle \Li_X Y,Z\rangle
	\end{align*}and hence 
	$\Li_X^\h Y-
	\Li_X Y\in\R\xi$. Moreover, since $\h$ is an ideal and by virtue of \eqref{Levi},
	\[ 2\langle \Li_X Y,\xi\rangle=\langle DX,Y\rangle+\langle DY,X\rangle. \]We deduce that
	\[ \Li_XY=\Li_X^\h Y+\frac12\left(\langle DX,Y\rangle_\h+\langle DY,X\rangle_\h \right)\xi,\quad X,Y\in\G. \]From the last relation in \eqref{kenmostubis} and since $\phi(\Li_XY)=\phi(\Li_X^\h Y)$, we get
	\begin{align*}
		0&=\Li_X\phi(Y)-\phi(\Li_XY)-\langle JX,Y\rangle_\h\xi\\
		&=\Li_X^\h JY+\frac12\left(\langle DX,JY\rangle_\h+\langle DJY,X\rangle_\h \right)\xi-J(\Li_X^\h Y)-\langle JX,Y\rangle_\h\xi.
	\end{align*}
	Thus
	\[ [\Li_X,J]=0\esp D+D^t=-2\mathrm{Id}_\h. \]	
	On the other hand,
	\[ 2\langle \Li_\xi X,Y\rangle=\langle DX,Y\rangle-\langle DY,X\rangle \]
	and hence
	\[ \Li_\xi X=\frac12(DX-D^tY),\quad X,Y\in\G. \]
	From the last relation in \eqref{kenmostubis},
	\[ \Li_\xi JX-J\Li_\xi X=0 \]and hence
	\[ [J,D-D^t]=0\esp [J,D]=0. \]
\end{proof}

Now, we get  all the needed ingredients to complete the proof of Theorem \ref{main}.

Let $(G,h,\phi,\xi)$ be a Kenmotsu Lie group of dimension $2n+1$. We have shown that $\G=\h\oplus\R\xi$ where $\h$ is an ideal having an invertible derivation $D=(\ad_\xi)_{|\h}$ and a complex isomorphism $J$ such that $(\h,\prs_\h,J)$ is a K\"ahler Lie algebra, $[J,D]=0$ and $D+D^t=-2\mathrm{Id}_\h$. According to \cite{Jacobson}, a Lie algebra with an invertible derivation must be nilpotent. On the other hand, a nilpotent K\"ahler Lie algebra must be abelian (see \cite{Benson}). Now, according to Lemma \ref{le}, there exists $(\la_1,\ldots,\la_n)\in\R^n$ and an orthonormal basis $(e_1,Je_1,\ldots,e_n,Je_n)$ of $\h$ such that, for any $j\in\{1,\ldots,n\}$,
\[ De_j=-e_j-\la_j Je_j\esp DJe_j=-Je_j+\la_j e_j. \] 
By comparing these Lie brackets to \eqref{bracket}, one can see that $(\G,\prs,\phi,\xi)$ is isomorphic to the Lie algebra  of the model $(G_\la,h_0,\phi_0,\xi_0)$. The fact that
$(G_\la,h_0)$ is Einstein and its Ricci curvature satisfies $\ric=-2n h_0$ can be checked by a straightforward computation. 
This completes the proof of Theorem \ref{main}.

\end{document}